\documentclass[12pt,reqno]{amsart}
\usepackage{amsmath,amssymb,amsfonts,amsthm,amscd,amstext,amsxtra,amsopn,array,url,verbatim,mathrsfs,enumerate,anysize,enumitem}
\usepackage{graphicx}
\usepackage{amsmath,amssymb,amsfonts,amsthm,amssymb,amscd,url,amstext,amsxtra,amsopn
,txfonts}
\usepackage{verbatim}
\marginsize{2.20cm}{2.20cm}{2.20cm}{2.20cm}
\usepackage{times}
\usepackage{amsmath}
\usepackage{amssymb}
\usepackage{amsfonts}
\usepackage{xcolor}
\usepackage{hyperref}

\usepackage{amsrefs}
\newenvironment{rezabib}
  {\bibdiv\biblist\setupbib}
  {\endbiblist\endbibdiv}
  
  \def\setupbib{\catcode`@=\active}
\begingroup\lccode`~=`@
  \lowercase{\endgroup\def~}#1#{\gatherkey{#1}}
\def\gatherkey#1#2{\gatherkeyaux{#1}#2\gatherkeyaux}
\def\gatherkeyaux#1#2,#3\gatherkeyaux{\bib{#2}{#1}{#3}}

\newtheorem{theorem}{Theorem}[section]

\newtheorem{proposition}[theorem]{Proposition}
\newtheorem{lemma}[theorem]{Lemma}

\newtheorem{corollary}[theorem]{Corollary}

\newtheorem{remarks}[theorem]{Remarks}
\numberwithin{equation}{section}

\newcommand{\nn}{\mathbb{N}}

\newcommand{\zz}{\mathbb{Z}}
\newcommand{\cc}{\mathbb{C}}
\newcommand{\qqq}{\mathbb{Q}}

\def\bals#1\nals{\begin{align*}#1\end{align*}}
\def\bal#1\nal{\begin{align}#1\end{align}}

\def \ds{\displaystyle}

\DeclareMathOperator{\SL}{SL}

\usepackage{mathtools}

\begin{document}

\begin{center}


\title[Bounds for orders of zeros of a class of Eisenstein series]{Bounds for orders of zeros of a class of Eisenstein series and their applications on dual pairs of eta quotients}

\author{Amir Akbary}
\address{Department of Mathematics and Computer Science, University of Lethbridge, Lethbridge, Alberta T1K 3M4, Canada}
\email{amir.akbary@uleth.ca}

\author{Zafer Selcuk Aygin}
\address{Science Department, Northwestern Polytechnic, Grande Prairie, AB T8V 4C4, Canada}%
\email{selcukaygin@gmail.com}%

\thanks{Research of the first author is partially supported by NSERC. Research of the second author is partially supported by a PIMS postdoctoral fellowship.}

\begin{abstract}
Let $k$ be an even positive integer, $p$ be a prime and $m$ be a nonnegative integer. We find an upper bound for orders of zeros (at cusps) of a linear combination of classical Eisenstein series of weight $k$ and level $p^m$. As an immediate consequence we find the set of all eta quotients that are linear combinations of these Eisenstein series and hence the set of all eta quotients of level $p^m$ whose derivatives are also eta quotients. 

\end{abstract}
\maketitle
\end{center}

\section{INTRODUCTION}\label{S1}
For  an element $z$ in the upper half-plane of the complex numbers, let $q:=e^{2\pi {\rm i}  z}$. The classical Eisenstein series are defined by
\bal
E_k(z):= \frac{-B_{k}}{2k} + \sum_{n \geq 1} \sigma_{k-1}(n) q^{n }, \label{exp6}
\nal
where $k\geq 2$ is an even integer, $B_k$ is the $k$-th Bernoulli number and $\sigma_{k-1}(n) = \sum_{0<d \mid n} d^{k-1}$. (Here we use the normalization given in \cite[p. 88]{stein} for the Eisenstein series $E_k(z)$.)          
Let 
\bals
\mathcal{E}_k(N) := \begin{cases}
\langle E_2(z)-dE_2(dz);~ 1< d\mid N \rangle_{\qqq} & \mbox{if $k=2$},\\
\langle E_k(dz);~ 1\leq d\mid N \rangle_{\qqq} & \mbox{otherwise.}
\end{cases}
\nals
Let $M_k(\Gamma_0(N))$ be the space of modular forms of weight $k$ for $\Gamma_0(N)$. Then it is known that for all even $k\geq 2$ the space $\mathcal{E}_k(N)$ is a subset of $M_k(\Gamma_0(N))$. 

Some infinite products that can be expressed explicitly as infinite sums are elements of $\mathcal{E}_k(N)$. For example we have
\bals
\prod_{n=1}^{\infty} \frac{(1-q^{2n})^{20}}{(1-q^{n})^{8}(1-q^{4n})^{8}}= 8E_2(z) -32E_2(4z),
\nals

\begin{multline*}
\prod_{n=1}^{\infty} \frac{(1-q^{2n})^{2}(1-q^{4n})^{4}(1-q^{6n})^{6}}{(1-q^{n})^{2}(1-q^{3n})^{2}(1-q^{12n})^{4}}\\= 2E_2(z)-3E_2(2z)+4E_2(4z)+9E_2(6z)-36E_2(12z),
\end{multline*}
and 
\begin{multline*}
\left(\prod_{n=1}^{\infty} \frac{(1-q^{2n})^{5}}{(1-q^{n})^{2}(1-q^{4n})^{2}}\right)^{4k}\\=
\frac{-2k}{B_{2k}}\left(\frac{(-1)^k}{2^{2k}-1} E_{2k}(z) - \frac{((-1)^k+1)}{2^{2k}-1} E_{2k}(2z) + \frac{2^{2k}}{2^{2k}-1} E_{2k}(4z) \right)\\
+ \sum_{n=1}^{\infty} O(n^{k}) q^{n},
\end{multline*}
where the first identity can be deduced from Jacobi's formula for the representation by four squares \cite{jacobi}, the second identity is from Williams's paper \cite[Table 1, No. 24]{williams-2}, and the last one is given by Ramanujan \cite[Sec. 25]{ramanujan} and proven by Mordell \cite{mordell}.

By using \cite[Corollary 2.1]{AT19} the first two indentities above induce the differential identities
\bals
D \left(q\prod_{n=1}^{\infty}\frac{(1-q^{4n})^{8}}{(1-q^{n})^{8}}\right) = q \prod_{n=1}^{\infty}\frac{(1-q^{2n})^{20}}{(1-q^{n})^{16}}
\nals
and
\begin{multline*}
D \left(q^2 \prod_{n=1}^{\infty} \frac{(1-q^{2n})^3 (1-q^{12n})^6 }{(1-q^{n})^4 (1-q^{4n})^2 (1-q^{6n})^3} \right) 
\\= 2 q^2 \prod_{n=1}^{\infty} \frac{(1-q^{2n})^5 (1-q^{4n})^2 (1-q^{6n})^3 (1-q^{12n})^2 }{(1-q^{n})^6 (1-q^{3n})^2}.
\end{multline*}
Here, $D:=q\frac{d}{dq}$ denotes the Ramanujan differential operator.
Additionally, the differential identity
\bals
E_2(z)^2&= \frac{5}{12} E_4(z) -\frac{1}{2} D( E_2(z))
\nals
appears in the works of Besge, Glaisher and Ramanujan independently (see \cite{williams} for the references). In all these examples, one can replace $z$ by $tz$ ($t \in \nn$) and obtain another product-to-sum formula. To avoid this triviality; and to avoid counting the same example more than once we define the set
\bals
\mathcal{P}_k(N) := \mathcal{E}_k(N) \backslash \mathcal{O}_k(N),
\nals 
where
\bals
\mathcal{O}_k(N) := \bigcup_{1<d \mid N} \left( \mathcal{E}_k(N/d) \cup \{ f(d z); f(z) \in \mathcal{E}_k(N/d) \}\right).
\nals

Let $R(N)$ denote a complete set of inequivalent cusps of $\Gamma_0(N)$ and for $r \in R(N)$, let $v_r(f)$ denote the order of vanishing of $f$ at the cusp $r$. Letting $f(z) \in \mathcal{P}_k(p^m)$, in this paper we find the following upper bound for the sum of orders of vanishings of $f(z)$ at cusps in $R(p^m)$. 
\begin{theorem}
\label{maingen}
Let $p$ be a prime and $m \in \nn_0 =\mathbb{N} \cup \{0\}$ and $k\geq 2$ be even. If $f(z) \in \mathcal{P}_k(p^m) $, then we have 
\bals
\sum_{r \in R(p^m)} v_r(f) <\begin{cases} 
1 & \mbox{if $p^m=1$,}\\
4 & \mbox{if $p^m=4$,}\\
p^{[(m-1)/2]} ( p^{(m-1)-2 [(m-1)/2]} +1) & \mbox{if $p^m \neq 1\mbox{ or } 4$}.
\end{cases}
\nals
\end{theorem}

\begin{remarks}
\begin{enumerate}
\item For $p^m \neq 4$, the proof of Theorem \ref{maingen} more generally establishes that if $f(z) \in \mathcal{P}_k(p^m) $, then $v_r(f) \leq 1$ at any cusp $r\in R(p^m)$. Note that 
$$\#R(p^m) = p^{[(m-1)/2]} (p^{(m-1)-2[(m-1)/2]}+1),$$
see \cite[Corollary 6.3.24.(b)]{CS17}.
\item The bounds given by Theorem \ref{maingen} do not depend on the weight $k$. Therefore, by valence formula, as the weight increases, the proportion of the zeros at the cusps compared to all the zeros decreases.
\item If we let the field of coefficients in the definition of $\mathcal{E}_k(N)$ to be $\cc$, then Theorem \ref{maingen} holds for
\bals
f(z)=\sum_{t \mid p^m} r_t E_k(tz) \in \mathcal{P}_k(p^m),
\nals
unless $m$ is even and  $ \omega \frac{r_{p^m}}{r_1} \neq - p^{mk/2}$ where $\omega$ is a certain $p^{m/2}$th root of unity defined in the proof of Lemma \ref{exp10}. In this case the bound on $\sum_{r \in R(p^m)} v_r(f)$ can be bigger because one of the arguments in the proof of Lemma \ref{exp10} may fail in certain cases. 
\end{enumerate}
\end{remarks}

We next describe an application of Theorem \ref{maingen}. The Dedekind eta function is defined by the infinite product 
$$\eta(z)=e^{\pi {\rm i} z/12} \prod_{n=1}^{\infty} (1-e^{2\pi {\rm i}   nz})= q^{1/24} \prod_{n=1}^{\infty} (1-q^n).$$
An \emph{eta quotient of level} $N$ is defined to be of the form
\begin{equation}
\label{eta}
f(z)=\prod_{t\mid N} \eta(tz)^{r_t},
\end{equation}
where the exponent $r_t$ are integers. The weight attached to this eta quotient is $k= \sum_{0<t\mid N} {r_t}/2$. Notice that we do not require level to be the lowest common multiple of $t$ for which $r_t \neq 0$, therefore the level of an eta quotient in our approach is not unique. This gives us a certain freedom in our discussions and does not affect the completeness of our results. 

We say an eta quotient $f$ is \emph{primitive} if there is no eta quotient $g$ such that $f(z)=g(dz)$ for some integer $d>1$. A pair $(f, g)$ of eta quotients is  called a \emph{dual pair} if $f$ is of weight $-k$ and $g$ is of weight $k+2$, for some non-negative integers $k$, and the ($k+1$)-th derivative of $f$ is a non-zero constant multiple of $g$. In \cite{choi} the problem of finding all dual pairs of eta quotients $(f, g)$ on $\Gamma_0(N)$ for which $f$ is a primitive eta quotients is studied. Theorem \ref{maingen}
has an immediate application on finding dual pairs of eta quotients $(f, g)$ of weight $(0, 2)$. These are
the eta quotients whose derivatives are also eta quotients (or constant multiple of eta quotients). 
In \cite{choi} the set of all such primitive eta quotients on $\Gamma_0(N)$ with squarefree levels $N$ is given; in \cite{AT19} a set of $203$ dual pairs of weight $(0, 2)$ was given and conjectured to be the complete set for all $N$. Additionally in \cite{AT19} it is established that every such pair is induced by the eta quotients in $\mathcal{E}_2(N)$. Since eta quotients have all their zeros (or poles) at cusps, as a direct consequence of our Theorem \ref{maingen} we establish the complete set of eta quotients of prime power levels whose derivatives are also eta quotients. This extends the results of \cite{choi} on dual pairs of weight $(0, 2)$ for squarefree levels to prime power levels.


As noted all zeros (or poles) of eta quotients are at the cusps. Hence, by the valence formula, for an eta quotient $f(z)$ of level $p^m$ we have $\sum_{r \in R(p^m)} v_r(f)= \frac{k}{12} ( p^{m} + p^{m-1})$ (see \cite[Lemma 2.1]{choi}). Therefore a comparison with the upper bound given by Theorem \ref{maingen} and investigations among possible pairs of $(k,p^m)$ give us the following statement.   
\begin{theorem}
\label{main}
Let $p$ be prime, $m \in \nn_0$ and $k \in \nn$ be even. Then there is no eta quotient in $\mathcal{P}_k(p^m)$ unless $(k,p^m)=(2,4)$, $(2,8)$, $(2,9)$, $(2,16)$, $(4,2)$ or $(4,4)$. 
\end{theorem}

In Corollary \ref{etaqlist} below the {\it trivial extensions} mean the following:
\begin{enumerate}
\item[] If $f(z) \in \mathcal{E}_k(N)$, then $cf(t_1z) \in \mathcal{E}_k(t_2 N)$ for all $t_2 \in \nn$, $t_1 \mid t_2$ and $c\in \qqq$.
\end{enumerate}

\begin{corollary} \label{etaqlist} Let $p$ be prime, $m \in \nn_0$.
\begin{enumerate}
\item The set of eta quotients 
\begin{multline*}
\left\{ \frac{\eta^8(z)}{\eta^4(z)}, \frac{\eta^8(4z)}{\eta^4(2z)}, \frac{\eta^{20}(2z)}{\eta^8(z)\eta^8(4z)}, \frac{\eta^4(z)\eta^{10}(4z)}{\eta^6(2z)\eta^4(8z)}, \frac{\eta^{10}(2z)\eta^4(8z)}{\eta^4(z)\eta^6(4z)}, 
\right.\\
\frac{\eta^6(2z)\eta^6(4z)}{\eta^4(z)\eta^4(8z)}, \frac{\eta^4(z)\eta^4(8z)}{\eta^2(2z)\eta^2(4z)}, 
 \frac{\eta^{10}(3z)}{\eta^3(z)\eta^3(9z)}, 
\frac{\eta^2(z)\eta^8(4z)\eta(8z)}{\eta^5(2z)\eta^2(16z)}, 
\\ \left.
\frac{\eta(2z)\eta^8(4z)\eta^2(16z)}{\eta^2(z)\eta^5(8z)}, \frac{\eta(2z)\eta^6(4z)\eta(8z)}{\eta^2(z)\eta^2(16z)}, \frac{\eta^2(z)\eta^{10}(4z)\eta^2(16z)}{\eta^5(2z)\eta^5(8z)} \right\}
\end{multline*}
is the complete set of eta quotients in $\mathcal{E}_2(p^m)$ (up to trivial extensions).
\item The set of eta quotients
\bals
\left\{ \frac{\eta^{16}(2z) }{\eta^{8}(z) }, \frac{\eta^{40}(2z) }{\eta^{16}(z) \eta^{16}(4z) }, \frac{\eta^{8}(z) \eta^{8}(4z) }{\eta^{8}(2z) }, \frac{\eta^{16}(z) }{\eta^{8}(2z) }  \right\}
\nals
is the complete set of eta quotients in $\mathcal{E}_4(p^m)$ (up to trivial extensions). 
\item If $k>4$, then there are no eta quotients in $\mathcal{E}_k(p^m)$.
\end{enumerate}
\end{corollary}


It is known that the eta quotients in $\mathcal{E}_2(N)$ are intimately related to the eta quotients of level $N$ whose derivatives are eta quotients. More precisely, Lemma 2.1 of \cite{AT19} establishes a one-to-one correspondence between the eta quotients of the form $\sum_{1<\delta\mid N} r_\delta \left(E_2(z)-\delta E_2(\delta z)\right)\in \mathcal{E}_2(N)$ and the eta quotients $\prod_{0<\delta\mid N} \eta^{r_\delta}(\delta z)$ with $r_1=-\sum_{1<\delta \mid N} r_\delta$ whose derivatives are also eta quotients.

\begin{corollary} \label{derlist}
Let $p$ be prime, $m \in \nn_0$. The set of eta quotients 
\begin{multline*}
 \left\{ \frac{\eta^{8}(z)\eta^{16}(4z)}{\eta^{24}(2z)},
\frac{\eta^{3}(2z)}{\eta^{2}(z)\eta(4z)},
\frac{\eta^{8}(4z)}{\eta^{8}(z)},
\frac{\eta^{4}(z)\eta^{2}(4z)\eta^{4}(8z)}{\eta^{10}(2z)},
\frac{\eta^{5}(4z)}{\eta^{2}(z)\eta(2z)\eta^{2}(8z)},\right.\\
 \left. \frac{\eta^{2}(2z)\eta^{4}(8z)}{\eta^{4}(z)\eta^{2}(4z)},
\frac{\eta^{7}(2z)\eta^{2}(8z)}{\eta^{2}(z)\eta^{7}(4z)},
\frac{\eta^{3}(9z)}{\eta^{3}(z)},
\frac{\eta^{2}(z)\eta^{2}(4z)\eta^{2}(16z)}{\eta^{5}(2z)\eta(8z)},\right.\\
 \left.
\frac{\eta(2z)\eta^{5}(8z)}{\eta^{2}(z)\eta^{2}(4z)\eta^{2}(16z)},
\frac{\eta(2z)\eta^{2}(16z)}{\eta^{2}(z)\eta(8z)},
\frac{\eta^{5}(2z)\eta^{2}(16z)}{\eta^{2}(z)\eta^{5}(8z)}
 \right\}
\end{multline*}
is the complete set (up to trivial extensions) of eta quotients of level $p^m$ whose derivatives are also eta quotients.
\end{corollary}

In Corollary \ref{derlist} the {\it trivial extensions} mean the following:
\begin{enumerate}
\item[]  If $f(z)$ is an eta quotient whose derivative is also an eta quotient (or a constant multiple of an eta quotient), then $f^{\ell}(tz)$ is an eta quotient whose derivative is also an eta quotient (or a constant multiple of an eta quotient) for all $t, \ell \in \nn$.
\end{enumerate}

There is a previously known example of an eta quotient whose second derivative is also an eta quotient 
\bals
D^2 \left( \frac{\eta^{2}(2z)}{\eta^4(z)} \right) = 4 \frac{\eta^{18}(2z)}{\eta^{12}(z)}.
\nals
From Corollary \ref{etaqlist}(2) we see that 
$$\left( {\frac{\eta^{18}(2z)}{\eta^{12}(z)}}\right) \bigg/ \left( {\frac{\eta^2(2z)} {\eta^{4}(z)}}\right) = \frac{\eta^{16}(2z)}{\eta^{8}(z)} \in \mathcal{E}_4(4).$$
This is not a coincidence. In fact we will show that any integer solution $(r_1, r_2, r_4)$ of the system
\bals
\begin{cases}
r_1+r_2+r_4=-2,\\
\frac{5}{12} r_1^2 + \frac{1}{3} r_1r_2=s_1,\\
\frac{5}{3} r_2^2 + \frac{4}{3} r_1 r_2 + \frac{1}{2} r_1 r_4  + \frac{4}{3} r_2 r_4=s_2,\\
\frac{20}{3} r_4^2  + \frac{8}{3} r_1 r_4  + \frac{16}{3} r_2 r_4 =s_4,
\end{cases}
\nals
where $\sum_{0<\delta\mid 4} s_\delta E_4(\delta z)\in \mathcal{E}_4(4)$, gives rise to an eta quotient of weight $-1$ and level $4$ for which its second derivative is an eta quotient. As a consequence of this we classify all the level 4 eta quotients whose second derivatives are also eta quotients.

\begin{theorem} \label{secordprop}
Up to trivial extensions the only level $4$ eta quotient whose second derivative is also an eta quotient is $\ds \frac{\eta^{2}(2z)}{\eta^4(z)}$.
\end{theorem}

The {\it trivial extensions} in Theorem \ref{secordprop} mean the following:
\begin{enumerate}
\item[] 
If $f(z)$ is an eta quotient whose second derivative is also an eta quotient (or a constant multiple of an eta quotient), then $f(tz)$ is an eta quotient whose second derivative is also an eta quotient (or a constant multiple of an eta quotient) for all $t \in \nn$.
\end{enumerate}

The arguments we use to prove Theorem \ref{secordprop} will not hold for eta quotients in most of the levels other than $4$ because formulas analogous to \eqref{conv1}--\eqref{conv3} in other levels almost always involve cusp forms of weight $4$. So the second derivative of an eta quotient in general may not be in $\mathcal{E}_4(N)$.

The organization of the rest of the paper is as follows. In the next section we derive the Fourier series expansions of $E_k(tz)$ at each cusp using the expansion at $i \infty$ of $E_k(z)$ and some matrix relations. We employ these expansions and use the description of the Eisenstein series in $\mathcal{P}_k(p^m)$ to prove Theorem \ref{main}. In Section \ref{sec3} we prove Theorems \ref{maingen} and \ref{main}, and Corollaries \ref{etaqlist} and \ref{derlist} by using the results of Section \ref{sec2}. Finally in Section \ref{sec4} we work on the second derivatives of eta quotients and prove Theorem \ref{secordprop}.




\section{The Fourier series expansions of $E_k(tz)$ at cusps $a/c$} \label{sec2}

Let $t \in \nn$. To prove Theorem \ref{maingen} we need to derive the Fourier series expansion of $E_{k}(tz)$ at $a/c$ where $a,c \in \zz$ satisfy $\gcd(a,c)=1$. We follow an approach similar to the proof of \cite[Proposition 2.1]{kohler} where constant terms of Dedekind eta function is calculated. For notational convenience let us denote
\bals
E_{k,t}(z):=E_k(tz) = \frac{-B_{k}}{2k} + \sum_{n \geq 1} \sigma_{k-1}(n) q^{tn}.
\nals
Define 
\bals
\iota_k := \begin{cases}
1 & \mbox{if $k=2$,}\\
0 & \mbox{otherwise.}
\end{cases}
\nals
Set
\bals
q_{c,N}:= e^{2 \pi {\rm i}  \gcd(c^2,N) z/N},
\nals
and
\bals
\omega_{M,t}:=\begin{cases}
1 & \mbox{if $c \equiv 0 \pmod{t}$,}\\
e^{(\frac{- 2 \pi {\rm i}  \gcd(t,c) df }{t})} & \mbox{if $c \not \equiv 0 \pmod{t}$,}
\end{cases}
\nals
where $M = \begin{pmatrix} a & b \\ c & d \end{pmatrix} \in \SL_2(\zz)$.
\begin{lemma} \label{expansionlemma}
Let $a/c$ be a rational cusp of $\Gamma_0(N)$, where $(a, c)=1$. Let $b$ and $d$ be two integers such that  $M = \begin{pmatrix} a & b \\ c & d \end{pmatrix} \in \SL_2(\zz)$. If $f(z) \in \mathcal{E}_k(N)$, then the Fourier series expansion of $f(z)=\sum_{t \mid N} r_t E_k(tz)$ at the cusp $a/c$ is given by
\begin{multline*}
(c z + d)^{-k} f(Mz) = (c z + d)^{-k} \sum_{t \mid N} r_t E_{k,t}(Mz) \\= \sum_{t \mid N} \sum_{n \geq 0} a_n(c,t) r_t \omega_{M,t}^n q_{c,N}^{ n \gcd(t,c)^2 N/ t \gcd(c^2,N) },
\end{multline*}
where $a_n(c, t) \neq 0$ are rational numbers given explicitly in the proof below.
\end{lemma}

\begin{proof}
It is known that
\bal
E_k(Mz)=E_k \left( \frac{az+b}{cz+d} \right)=(cz+d)^k E_k(z) +\iota_k \frac{{\rm i} c}{ 4\pi} (cz+d) \label{exp4}
\nal
for all $M = \begin{pmatrix} a & b \\ c & d \end{pmatrix} \in \SL_2(\zz)$ (see \cite[Corollary 5.2.17.(b)]{CS17}). We have
\bal
E_{k,t}(Mz) = E_k (t M z) = E_k \left( \frac{at z + bt}{cz+d} \right). \label{exp1}
\nal
Now let $e := \frac{at}{\gcd(t,c)}$ and $g := \frac{c}{\gcd(t,c)}$. It is clear that $e,g \in \zz$ and since $\gcd(e,g)=1$, there exists   $f,h \in \zz$ such that
\bal
& eh-fg=1. \label{exp2}
\nal
Therefore $\begin{pmatrix} e & f \\ g & h \end{pmatrix} \in \SL_2(\zz)$. Hence using \eqref{exp2} we obtain
\bal
\begin{pmatrix} at & bt \\ c & d \end{pmatrix} = \begin{pmatrix} e & f \\ g & h \end{pmatrix}  \begin{pmatrix} a h t - c f & bht - df \\ 0 & -bgt + de \end{pmatrix}.  \label{exp3}
\nal
Putting \eqref{exp1} and \eqref{exp3} together we obtain
\bals
E_{k,t}(Mz)= E_k \left( \begin{pmatrix} e & f \\ g & h \end{pmatrix}  \begin{pmatrix} a h t - c f & bht - df \\ 0 & -bgt + de \end{pmatrix} z \right).
\nals
Let
\bals
z_{M} := \frac{(a h t - c f) z + (bht - df) }{-bgt + de}.
\nals
Then using $a h t - c f = \gcd(t,c)$ we get
\begin{equation}
\label{gz}
gz_M+h 
  = \frac{\gcd(t,c)}{t} (c z + d).
\end{equation}
Observe that
\begin{equation}
\label{bg}
-bgt + de=t/\gcd(t,c).
\end{equation}
Now, by employing  \eqref{gz}, \eqref{bg}, and \eqref{exp4} we have
\bals
E_{k,t}(Mz) 
 = \left(\frac{\gcd(t,c)}{t} \right)^k (c z + d)^k E_k\left( \frac{\gcd(t,c)^2 z - \gcd(t,c) df }{t} \right) +\iota_k \frac{{\rm i} c}{ 4\pi t} (cz+d),
 \nals
where in the calculations we use \eqref{exp2}
and
\bals
E_2(z+1)=E_2(z).
\nals
Additionally, if $c \equiv 0 \pmod{t}$, then $\gcd(t,c) df \equiv 0 \pmod{t}$. That is, we have
\bal
E_{k,t}(Mz)=\begin{cases}
\left(\frac{\gcd(t,c)}{t} \right)^k (c z + d)^k E_k\left( \frac{\gcd(t,c)^2 z}{t} \right) & \\
\quad 
+\iota_k \frac{{\rm i} c}{ 4\pi t} (cz+d),
& \mbox{if $c \equiv 0 \pmod{t}$,}\\
\left(\frac{\gcd(t,c)}{t} \right)^k (c z + d)^k E_k\left(\frac{\gcd(t,c)^2 z - \gcd(t,c) df }{t} \right) & \\
\quad
+\iota_k \frac{{\rm i} c}{ 4\pi t} (cz+d),
& \mbox{if $c \not\equiv 0 \pmod{t}.$}
\end{cases} \label{exp5}
\nal
Applying \eqref{exp6} in \eqref{exp5} we obtain the following Fourier series expansion of $E_k(tz)$ at $a/c$:
\begin{multline}
E_{k,t}(Mz)\\
=\begin{cases}
\left(\frac{\gcd(t,c)}{t} \right)^k (c z + d)^k \left( \frac{-B_k}{2k} + \sum_{n \geq 1} \sigma_{k-1}(n) e^{2 \pi {\rm i}  n \gcd(t,c)^2 z/t} \right) & \\
\quad 
+ \iota_k \frac{{\rm i} c}{ 4\pi t} (cz+d),
& \hspace{-1.5cm} \mbox{if $c \equiv 0 \pmod{t}$,}\\
\left(\frac{\gcd(t,c)}{t} \right)^k (c z + d)^k \left( \frac{-B_k}{2k} + \sum_{n \geq 1} \sigma_{k-1}(n) e^{2 \pi {\rm i}  n (\frac{\gcd(t,c)^2 z - \gcd(t,c) df }{t})}    \right) & \\
\quad  
+\iota_k \frac{{\rm i} c}{ 4\pi t} (cz+d),
& \hspace{-1.5cm} \mbox{if $c \not\equiv 0 \pmod{t}.$}
\end{cases} \label{exp7}
\end{multline}
Note that the width of the cusp $a/c$ in $\Gamma_0(N)$ is given by $\frac{N}{\gcd(c^2,N)}$ (\cite[Corollary 6.3.24.(a)]{CS17}) and recall that we define 
\bals
q_{c,N}= e^{2 \pi {\rm i}  \gcd(c^2,N) z/N}.
\nals
Employing this in \eqref{exp7}, we obtain
\begin{multline}
(c z + d)^{-k} E_{k,t}(Mz)\\
=\begin{cases}
\sum_{n \geq 0} a_n q_{c,N}^{ n \gcd(t,c)^2 N/ t \gcd(c^2,N) } 
+\iota_k \frac{ {\rm i} c}{  4 \pi t (c z + d)}
& \mbox{if $c \equiv 0 \pmod{t}$,}\\
\sum_{n \geq 0} a_n e^{(\frac{- 2 \pi {\rm i}  n \gcd(t,c) df }{t})} q_{c,N}^{ n \gcd(t,c)^2 N /t\gcd(c^2,N)  }   
+\iota_k \frac{ {\rm i} c}{  4 \pi t (c z + d)}
& \mbox{if $c \not\equiv 0 \pmod{t},$}
\end{cases} \label{exp8}
\end{multline}
where, for a fixed $k$, $a_n$ depends on $c$ and $t$ and given by
\bal
a_n:=a_n (c,t)=\begin{cases}
\ds \left(\frac{\gcd(t,c)}{t} \right)^k \cdot \frac{-B_k}{2k} & \mbox{if $n=0$,}\\
\ds \left(\frac{\gcd(t,c)}{t} \right)^k  \sigma_{k-1}(n) & \mbox{if $n>0$,}
\end{cases} \label{anval}
\nal
and hence $a_n(c,t)  \in \qqq$ and $a_n(c,t) \neq 0$ for all $n \in \nn_0$. Now recall that
\bals
\omega_{M,t}=\begin{cases}
1 & \mbox{if $c \equiv 0 \pmod{t}$,}\\
e^{(\frac{- 2 \pi {\rm i}  \gcd(t,c) df }{t})} & \mbox{if $c \not \equiv 0 \pmod{t}$.}
\end{cases}
\nals
Using this and noting that $a_n(c,t)$ is a function of $c$, $t$ and $k$, we write \eqref{exp8} as
\bal
(c z + d)^{-k} E_{k,t}(Mz) = \sum_{n \geq 0} a_n(c,t) \omega_{M,t}^n q_{c,N}^{ n \gcd(t,c)^2 N/ t \gcd(c^2,N) }
+\iota_k \frac{ {\rm i} c}{  4 \pi t (c z + d)}.
 \label{mar01}
\nal

Now, let $f(z) \in \mathcal{E}_k(N)$, then we have
\bal
f(z)=\sum_{t \mid N} r_t E_k(tz), \label{mar02}
\nal
where, when $k=2$, $r_t$ satisfy 
\bal
\sum_{t \mid N} \frac{r_t}{t} = 0, \label{e2q}
\nal
and $\iota_k=0$ for all even $k \geq 4$, that is, for all even $k \geq 2$ we have
\bals
\iota_k \sum_{t \mid N} r_t \frac{{\rm i}  c}{ 4\pi t (c z + d)} = \iota_k \frac{ {\rm i}  c}{4 \pi (c z + d)} \sum_{t \mid N} \frac{r_t}{t} = 0.
\nals
Therefore from \eqref{mar01} and \eqref{mar02} we have the desired result.
\end{proof}

\section{Proof of Theorem \ref{maingen}} \label{sec3}
In this section we let $N=p^m$ for $p$ a prime and $m$ a positive integer. Then for any $f(z)  \in \mathcal{E}_k(p^m)$ we have
\bals
f(z) = \sum_{t \mid p^m} r_t E_{k,t}(z),
\nals
where $r_t \in \qqq $ and should satisfy \eqref{e2q} when $k=2$. Let $M = \begin{pmatrix} a & b \\ c & d \end{pmatrix} \in \SL_2(\zz)$. Then for all $c \mid p^m$, by Lemma \ref{expansionlemma}, we have
\bals
(c z + d)^{-k} f(Mz) = \sum_{t \mid p^m } \sum_{n \geq 0} a_n(c,t) r_t \omega_{M,t}^n q_{c,p^m}^{ n \gcd(t,c)^2 p^m/ t \gcd(c^2,p^m) }.
\nals
Since all divisors of $p^m$ are of the form $p^i$, for $0 \leq i \leq m$ we have 
\bal
(p^i z + d)^{-k} f(Mz) = \sum_{ j= 0}^{m} \sum_{n \geq 0} a_n(p^i,p^j) r_{p^j} \omega_{M,p^j}^n q_{p^i,p^m}^{ n \gcd(p^j,p^i)^2 p^m/ p^j \gcd(p^{2i},p^m) }. \label{exp9}
\nal
On the other hand we have
\bals
\omega_{M,p^j} & =\begin{cases}
1 & \mbox{if $p^i \equiv 0 \pmod{p^j}$}\\
e^{- 2 \pi {\rm i} \gcd(p^j,p^i) d f/p^{j}} & \mbox{if $p^i \not \equiv 0 \pmod{p^j}$}
\end{cases}
 =\begin{cases}
1 & \mbox{if $i \geq j$,}\\
e^{- 2 \pi {\rm i}  d f/p^{j-i}} & \mbox{if $i < j$,}
\end{cases}
\nals
and
\bals
\gcd(p^j,p^i)^2 p^m/ p^j \gcd(p^{2i},p^m) 
 =\begin{cases}
p^{j} & \mbox{if $i \geq j$ and $i \geq m/2$,}\\
p^{2i-j} & \mbox{if $i < j$ and $i \geq m/2$,}\\
p^{m+j-2i} & \mbox{if $i \geq j$ and $i < m/2$,}\\
p^{m-j} & \mbox{if $i < j$ and $i < m/2$.}\\
\end{cases}
\nals
We put these into \eqref{exp9} to obtain
\bal
(p^i z + d)^{-k} f(Mz) 
= \begin{cases}
\ds \sum_{i \geq j} \sum_{n \geq 0} a_n(p^i,p^j) r_{p^j} q_{p^i,p^m}^{ np^{j}} & \\
\ds + \sum_{ i < j} \sum_{n \geq 0} a_n(p^i,p^j) r_{p^j} e^{- 2 \pi {\rm i}  n d f/p^{j-i}}  q_{p^i,p^m}^{ n p^{2i-j}} & \mbox{if $i \geq m/2$,}\\
\ds \sum_{ i \geq j} \sum_{n \geq 0} a_n(p^i,p^j) r_{p^j} q_{p^i,p^m}^{ np^{m+j-2i}} & \\
\ds + \sum_{ i < j} \sum_{n \geq 0} a_n(p^i,p^j) r_{p^j} e^{- 2 \pi {\rm i}  n d f/p^{j-i}}  q_{p^i,p^m}^{ np^{m-j}} & \mbox{if $i < m/2$.}\\
\end{cases} \label{eqexpf}
\nal

\begin{lemma}\label{exp10}
Let $f(z) \in \mathcal{E}_k(p^m)$ where $p^m \neq 4$. If $v_{a/p^i}(f(z))>1$ for some $0 \leq i \leq m$, then either $r_1 =0$ or $r_{p^m} =0$.
\end{lemma}

\begin{proof}
Let $f(z) \in \mathcal{E}_k(p^m)$. Let $0 \leq i < m/2$. By \eqref{eqexpf} the Fourier series expansion of $f(z)$ at $a/p^i$ is given by
\begin{multline*}
(p^i z + d)^{-k} f(Mz) =
\ds \sum_{ i \geq j} \sum_{n \geq 0} a_n(p^i,p^j) r_{p^j} q_{p^i,p^m}^{ np^{m+j-2i}} \\+ \sum_{ i < j} \sum_{n\geq 0} a_n(p^i,p^j) r_{p^j} e^{- 2 \pi {\rm i}  n d f/p^{j-i}}  q_{p^i,p^m}^{ np^{m-j}}.
\end{multline*}
If $v_{a/p^i}(f(z))>1$, then the coefficient of $q_{p^i,p^m}^1$ in the Fourier series expansion of $f(z)$ at $a/p^i$ is $0$, i.e., $a_1(p^i,p^m) r_{p^m} e^{- 2 \pi {\rm i}  n d f/p^{m-i}}=0$. Since $a_1(p^i,p^m)e^{- 2 \pi {\rm i}  n d f/p^{m-i}} \neq 0$, we must have $r_{p^m}=0$. Arguing similarly we obtain
\bals
r_{p^m}& =0 & \mbox{ if $i < m/2$,}\\
r_1& =0 & \mbox{ if $i > m/2$,}\\
a_1(p^{m/2},1) r_1 + a_1(p^{m/2},p^{m}) e^{-2 \pi {\rm i}  d f/p^{m/2}}r_{p^m}& =0 &  \mbox{ if $i = m/2$.}
\nals
Since, in notation of Lemma \ref{expansionlemma}, $c=p^{m/2}$, $t=p^m$, $ad-bc=1$, and $eh-fg=1$, we have $\gcd(df,p)=1$. Therefore when $p\neq 2$ and $m \neq 2$, $e^{-2 \pi {\rm i}  d f/p^{m/2}}$ is not a real number. Now since $a_1(p^{m/2},1)$, $ a_1(p^{m/2},p^{m}) \in \qqq$ we have $r_{1}=0$ and $r_{p^m}=0$. Therefore the statement follows.



\end{proof}




\begin{proof}[Proof of Theorem \ref{maingen}] Assume $p^m \neq 4$. Let $f(z) \in \mathcal{P}_k(p^m) $. Assume, for sake of contradiction, that $v_{a/p^i}(f(z)) > 1$ for some cusp $a/p^i$ where $0 \leq i \leq m$ and $\gcd(a,p^i)=1$. Then by Lemma \ref{exp10} we have either $r_1=0$ or $r_{p^{m}}=0$. If $r_1=0$, then there exists a $g(z) \in \mathcal{E}_k(p^{m-1})$ such that $f(z)=g(pz)$, that is, $f(z) \in \mathcal{O}_k(p^m)$. This contradicts to $f(z)$ being in $\mathcal{P}_k(p^m)$. If $r_{p^m}=0$, then $f(z) \in \mathcal{E}_k(p^{m-1})$ again a contradiction with $f(z)$ being in $\mathcal{P}_k(p^m)$. Therefore we must have $v_{a/p^i}(f(z)) \leq 1$ for all cusps $a/p^i$ where $0 \leq i \leq m$ and $\gcd(a,p^i)=1$. Since $f$ is not a cusp form, for at least one $r \in R(p^m)$ we have $v_r(f(z))=0$. Then the theorem follows from observing
\bals
\# R(p^m) = p^{[(m-1)/2]} ( p^{(m-1)-2 [(m-1)/2]} +1), 
\nals
see \cite[Corollary 6.3.24.(b)]{CS17}, together with the fact that all $r \in R(p^m)$ can be chosen to be in the form $\ds {a}/{p^i}$.

Now we prove the case $p=2$ and $m=2$. In this case let us fix \bals
R(4)=\left\{ \frac{1}{1}, \frac{1}{2}, \frac{1}{4} \right\}.
\nals
By arguments similar to the proof of Lemma \ref{exp10} if $v_{1/1}(f(z))>1$ or $v_{1/4}(f(z))>1$, then $r_1=0$ or $r_4=0$, that is, $f(z)$ cannot be in $\mathcal{P}_k(4)$. Therefore, we must have $v_{1/1}(f(z)) \leq 1$ and $v_{1/4}(f(z)) \leq 1$. On the other hand if $v_{1/2}(f(z))>2$, then from the expansion at $1/2$, formula \eqref{eqexpf} for $p=2$, $m=2$, $i=1$, we have
\begin{multline*}
(2 z + d)^{-k} f(Mz) = \ds \sum_{n \geq 0} a_n(2,1) r_{1} q_{2,4}^{ n} \\
+ \sum_{n \geq 0} a_n(2,2) r_{2} q_{2,4}^{ 2n} + \sum_{n\geq 0} a_n(2,4) r_{4} e^{- 2 \pi {\rm i}  n d f/2}  q_{2,4}^{ n}.
\end{multline*}
Observing that $df\equiv 1 \pmod{2}$ 
yields
\bals
a_0(2,1) r_1 + a_0(2,2) r_2 + a_0(2,4) r_4 &=0, \\ 
a_1(2,1) r_1 - a_1(2,4) r_4 &=0,\\ 
a_2(2,1) r_1 + a_1(2,2) r_2 + a_2(2,4) r_4 &=0.
\nals
Putting the values of $a_n$'s from \eqref{anval} in the the above equations and simplifying them we have
\bals
r_1 + r_2 + r_4/2^k &=0,\\ 
r_1 - r_4/2^k &=0,\\ 
(1+2^{k-1}) r_1 + r_2 + (1+2^{k-1}) r_4/2^k &=0.
\nals
Thus $r_1,r_2,r_4=0$. Therefore, if $v_{1/2}(f(z))>2$, then $f(z)=0$. 

Hence, if $f(z) \in \mathcal{P}_k(4)$, then $v_{1/1}(f(z)) \leq 1$, $v_{1/4}(f(z)) \leq 1$ and $v_{1/2}(f(z)) \leq 2$. Also since $f(z)$ is not a cusp form $v_r(f(z))$ has to be $0$ for at least one $r \in R(4)$. Thus, we have $\sum_{r \in R(4)} v_r(f(z)) < 4$. 
\end{proof}

\begin{proof}[Proof of Theorem \ref{main}]

Let $f \in \mathcal{P}_k(p^m)$ be an eta quotient. Then we have
\bals
\sum_{r \in R(p^m)} v_r(f)= \begin{cases}
\displaystyle \frac{k}{12} & \mbox{ if $m = 0$,}\\
\displaystyle \frac{k}{12} ( p^{m} + p^{m-1}) & \mbox{ if $m \neq 0$,}
\end{cases}
\nals 
see \cite[Lemma 2.1]{choi}. 

We observe that $\sum_{r \in R(p^m)} v_r(f)$ is greater than or equal to the upper bound given by Theorem \ref{maingen} unless $(k,p^m)=(2,1)$, $(4,1)$, $(6,1)$, $(8,1)$, $(10,1)$, $(2,2)$,  $(2,4)$, $(2,8)$, $(2,16)$, $(4,2)$, $(4,4)$, $(6,2)$, $(6,4)$, $(2,3)$, $(2,9)$, $(4,3)$, $(2,5)$, $(2,25)$, $(2,7)$. This means unless $(k,p^m)$ is one of the pairs above, $f \in \mathcal{P}_k(p^m)$ will have at least one zero in the upper half-plane, that is, $f$ cannot be an eta quotient. We run the algorithms given in \cite{AT19} and see that there are no eta quotients in $\mathcal{P}_k(p^m)$ when $(k,p^m)=(2,1)$, $(4,1)$, $(6,1)$, $(8,1)$, $(10,1)$, $(2,2)$, $(6,2)$, $(6,4)$, $(2,3)$, $(4,3)$, $(2,5)$, $(2,25)$, $(2,7)$. 
\end{proof}


\begin{proof}[Proof of Corollary  \ref{etaqlist}]
Notice that for any $f\in \mathcal{E}_k(p^m)$ there is a $g\in \mathcal{P}_k(p^{m^\prime})$ such that $0\leq m^\prime \leq m$ and $f$ is a trivial extension of $g$. By Theorem \ref{main}, we have $$(k, p^{m^\prime})\in \{(2, 4), (2, 8), (2, 9), (2, 16), (4, 2), (4, 4)\}.$$ Now we can derive the finite set of all the eta quotients of weight $k$ and level $p^{m^\prime}$ in $\mathcal{P}_k(p^{m^\prime})$ by employing Algorithm 4.2 in \cite{AT19}.
\end{proof}

\begin{proof}[Proof of Corollary  \ref{derlist}] By Lemma 2.1 of \cite{AT19} the desired set can be obtained by finding the antiderivatives of the eta quotients given in part (1) of Corollary \ref{etaqlist}. Now Algorithm 4.3 of \cite{AT19} derives the required antiderivatives.
\end{proof}

\section{Second order derivatives} \label{sec4}
In this section we prove Theorem \ref{secordprop}.  We start by determining the weight of an eta quotient whose second derivative is also an eta quotient. 

\begin{proposition}
\label{2nd-der}
If $f$ is an eta quotient of weight $k$ for which $D^2(f)/ f$ is an eta quotient of weight $\ell$, then $k=-1$ and $\ell=4$.
\end{proposition}

The following lemma is needed for the proof of this proposition.

\begin{lemma}
Let $f$ be a nonzero holomorphic function defined on the upper half-plane. Suppose that for an $M= \begin{pmatrix} a & b \\ c & d \end{pmatrix} \in \SL_2(\mathbb{Z})$ we have
\begin{equation}
\label{transformation}
f(Mz)=\psi_{M} (cz+d)^k f(z),
\end{equation}
where $k$ is an integer and $\psi_{M}$ is a root of unity depending on $M$. Then
\begin{equation}
\label{second-der}
(cz+d)^{-4}\frac{D^2(f(Mz))}{f(Mz)}=\frac{D^2(f(z))}{f(z)}+ \frac{k+1}{\pi {\rm i}} \frac{D(f(z))}{f(z)} \left(\frac{c}{cz+d} \right)-\frac{k(k+1)}{4\pi^2} \left( \frac{c}{cz+d} \right)^2.
\end{equation}

\end{lemma}
\begin{proof}
This is a consequence of  differentiation of \eqref{transformation} twice and employing the transformation property \eqref{transformation}. (For a general formula for the $m$-th derivative of a modular form see \cite[Proposition 3.1]{R12}.) 
\end{proof}

\begin{proof}[Proof of Proposition \ref{2nd-der}]
Since  $f$ is an eta quotient of weight $k$, then \eqref{transformation} holds for any $M\in \Gamma_0(R)$ for a suitable non-negative integer $R$ and  for corresponding $24$-th roots of unity $\psi_M$'s. Thus \eqref{second-der} holds for such $M$'s. Now assume that $D^2(f)/f$ is also an eta quotient of weight $\ell$. Suppose that this eta quotient has level $N$ (a multiple of $R$) and thus satisfies 
\begin{equation}
\label{transformation-2}
\frac{D^2(f(Mz))}{f(Mz)}=\chi_M (cz+d)^\ell \frac{D^2(f(z))}{f(z)}
\end{equation}
for some $24$-th root of unity $\chi_M$, depending on $M$,  and for all matrices $M= \begin{pmatrix} a & b \\ c & d \end{pmatrix} \in \Gamma_0(N)$. Now, for a fixed $z$ in the upper half-plane,  applying \eqref{transformation-2} in \eqref{second-der} and rearranging the terms yields
\begin{equation}
\label{main44}
\left(\chi_M (cz+d)^\ell-(cz+d)^4\right) \frac{D^2(f(z))}{f(z)}=(k+1) \left(kc^2 \alpha (cz+d)^2+c\beta(cz+d)^3  \right)
\end{equation}
for some fixed constants $\alpha, \beta \in \mathbb{C}$. (Note that since $z$ is fixed, the value of $D(f(z))/f(z)$ is absorbed in $\beta$.) Next assume that $\ell\geq 0$ and $\ell\neq 4$. Observe that the set $S$ of matrices  $M= \begin{pmatrix} * & * \\ N & d \end{pmatrix} \in \Gamma_0(N)$ for which $\chi_0 (Nz+d)^\ell-(Nz+d)^4\neq 0$, for one of the $24$-th roots of unity $\chi_0$, is infinite. Assuming that $k\neq -1$, the equation \eqref{main44}, for $c=N$ and $M\in S$, can be re-written as 
\begin{equation}
\label{main-2}
 \frac{1}{k+1}\frac{D^2(f(z))}{f(z)}=\frac{kN^2 \alpha (Nz+d)^2+N\beta(Nz+d)^3 }{\chi_0 (Nz+d)^\ell-(Nz+d)^4}.
\end{equation}
Now for a fixed $z$ in the upper half-plane, the right-hand side of \eqref{main-2} is equal to a fixed non-zero complex number $\gamma$ for infinitely many values of $d$. This is a contradiction as the non-trivial polynomial equation 
$$(kN^2\alpha) X^2+(N\beta) X^3-(\gamma\chi_0) X^\ell+\gamma X^4=0$$
has finitely many solutions. The other cases can be analyzed in a similar fashion to conclude that if both $f$ and $D^2(f)/f$ are eta quotients, then 
$k=-1$, $\ell=4$, and $\chi_M=1$. 
\end{proof}

We continue
with listing some identities that will be useful in the proof:
\bal
E_2(z)^2&= \frac{5}{12} E_4(z) -\frac{1}{2} D(E_2(z)), \label{conv1}
\nal
\bal
E_2(z)E_2(2z) = \frac{1}{12}E_4(z) + \frac{1}{3} E_4(2z) - \frac{1}{8} D(E_2(z)) - \frac{1}{4} D(E_2(2z)), \label{conv2}
\nal
and
\bal
E_2(z)E_2(4z) = \frac{1}{48}E_4(z) + \frac{1}{16} E_4(2z)  + \frac{1}{3} E_4(4z) - \frac{1}{16} D( E_2(z)) - \frac{1}{4} D( E_2(4z)), \label{conv3}
\nal
first of which is due to Besge, Glaisher and Ramanujan independently and for the latter two see \cite[Theorems 2 and 4]{williams}.
Additionally we note that using \eqref{conv1} and replacing $z$ by $2z$ we obtain
\bal
E_2(2z)^2= \frac{5}{12} E_4(2z) -\frac{1}{4} D( E_2(2z)),
\nal
using \eqref{conv1} and replacing $z$ by $4z$ we obtain
\bal
E_2(4z)^2= \frac{5}{12} E_4(4z) -\frac{1}{8} D( E_2(4z)),
\nal
and using \eqref{conv2} and replacing $z$ by $2z$ we obtain
\bal
E_2(2z)E_2(4z) = \frac{1}{12}E_4(2z) + \frac{1}{3} E_4(4z) - \frac{1}{16} D( E_2(2z)) - \frac{1}{8} D( E_2(4z)). \label{convlast}
\nal

Now we prove Theorem \ref{secordprop}.

\begin{proof}[Proof of Theorem \ref{secordprop}]
If $f(z) = \eta^{r_1}(z) \eta^{r_2}(2z)\eta^{r_4}(4z)$, then by taking logarithmic derivative two times and employing the identity $$D (\log(\eta(z)))=-E_2(z)$$
we obtain
\begin{multline*}
\frac{ D^2 (f(z))}{f(z)} 
=D (- r_1 E_2(z) - 2 r_2 E_2(2z) - 4 r_4 E_2(4z)) \\+  (- r_1 E_2(z) - 2 r_2 E_2(2z) - 4 r_4 E_2(4z))^2.
\end{multline*}
Next we use \eqref{conv1}--\eqref{convlast} and obtain
\bals
\frac{ D^2( f(z))}{f(z)} & = E_4(z) \left( \frac{5}{12} r_1^2 + \frac{1}{3} r_1r_2 \right) + E_4(2z) \left( \frac{5}{3} r_2^2 + \frac{4}{3} r_1 r_2 + \frac{1}{2} r_1 r_4  + \frac{4}{3} r_2 r_4  \right)\\
& \quad  + E_4(4z) \left(\frac{20}{3} r_4^2  + \frac{8}{3} r_1 r_4  + \frac{16}{3} r_2 r_4  \right)\\
& \quad + D(E_2(z)) \left( -r_1 - \frac{1}{2}r_1 (r_1 + r_2 + r_4)  \right)\\
& \quad + D(E_2(2z)) \left( -2r_2 - r_2(r_1 + r_2 +  r_4) \right)\\
& \quad + D(E_2(4z)) \left( -4r_4 - 2 r_4 (r_1 + r_2 + r_4)  \right).
\nals
Noting that, by Proposition \ref{2nd-der},  $r_1 + r_2 + r_4=-2$ must hold, we have
\bals
\frac{ D^2( f(z))}{f(z)} & = E_4(z) \left( \frac{5}{12} r_1^2 + \frac{1}{3} r_1r_2 \right) \\
& \quad + E_4(2z) \left( \frac{5}{3} r_2^2 + \frac{4}{3} r_1 r_2 + \frac{1}{2} r_1 r_4  + \frac{4}{3} r_2 r_4  \right)\\
& \quad  + E_4(4z) \left(\frac{20}{3} r_4^2  + \frac{8}{3} r_1 r_4  + \frac{16}{3} r_2 r_4  \right).
\nals
Therefore $\frac{ D^2( f(z))}{f(z)} \in \mathcal{E}_4(4)$. The result follows by investigating all the eta quotients in $\mathcal{E}_4(4)$ that are found by using the algorithms given in \cite{AT19}. 
\end{proof}

\section*{Acknowledgement}

The authors would like to thank the referee for valuable comments and suggestions.

\begin{rezabib}


\bib{AT19}{article}{
   author={Aygin, Zafer Selcuk},
   author={Toh, Pee Choon},
   title={When is the derivative of an eta quotient another eta quotient?},
   journal={J. Math. Anal. Appl.},
   volume={480},
   date={2019},
   number={1},
   pages={123366, 22},
   issn={0022-247X},
   review={\MR{3994912}},
   doi={10.1016/j.jmaa.2019.07.056},
}	

\bib{choi}{article}{
   author={Choi, Dohoon},
   author={Kim, Byungchan},
   author={Lim, Subong},
   title={Pairs of eta-quotients with dual weights and their applications},
   journal={Adv. Math.},
   volume={355},
   date={2019},
   pages={106779, 51},
   issn={0001-8708},
   review={\MR{3996729}},
   doi={10.1016/j.aim.2019.106779},
}

\bib{CS17}{book}{
   author={Cohen, Henri},
   author={Str\"{o}mberg, Fredrik},
   title={Modular forms, a classical approach},
   series={Graduate Studies in Mathematics},
   volume={179},
   publisher={American Mathematical Society, Providence, RI},
   date={2017},
   pages={xii+700},
   isbn={978-0-8218-4947-7},
   review={\MR{3675870}},
   doi={10.1090/gsm/179},
}

\bib{williams}{article}{
   author={Huard, James G.},
   author={Ou, Zhiming M.},
   author={Spearman, Blair K.},
   author={Williams, Kenneth S.},
   title={Elementary evaluation of certain convolution sums involving
   divisor functions},
   conference={
      title={Number theory for the millennium, II},
      address={Urbana, IL},
      date={2000},
   },
   book={
      publisher={A K Peters, Natick, MA},
   },
   date={2002},
   pages={229--274},
   review={\MR{1956253}},
}

\bib{jacobi}{book}{
   author={Jacobi, C. G. J.},
   title={Fundamenta Nova Theoriae Functionum Ellipticarum},
   series={Reprinted in Gesammelte Werke},
   volume={I},
   publisher={Chelsea, New York},
   date={1969},
}
\bib{kohler}{book}{
   author={K\"{o}hler, G\"{u}nter},
   title={Eta products and theta series identities},
   series={Springer Monographs in Mathematics},
   publisher={Springer, Heidelberg},
   date={2011},
   pages={xxii+621},
   isbn={978-3-642-16151-3},
   review={\MR{2766155}},
   doi={10.1007/978-3-642-16152-0},
}

\bib{mordell}{article}{
   author={Mordell, L. J.},
   title={On the representations of numbers as a sum of $2r$ squares},
   journal={Quart. J. Pure Appl. Math.},
   volume={48},
   date={1917},
   pages={93--104},
}


\bib{ramanujan}{article}{
   author={Ramanujan, S.},
   title={On certain arithmetical functions [Trans. Cambridge Philos. Soc.
   {\bf 22} (1916), no. 9, 159--184]},
   conference={
      title={Collected papers of Srinivasa Ramanujan},
   },
   book={
      publisher={AMS Chelsea Publ., Providence, RI},
   },
   date={2000},
   pages={136--162},
   review={\MR{2280861}},
   doi={10.1016/s0164-1212(00)00033-9},
}

\bib{R12}{article}{
   author={Royer, Emmanuel},
   title={Quasimodular forms: an introduction},
   language={English, with English and French summaries},
   journal={Ann. Math. Blaise Pascal},
   volume={19},
   date={2012},
   number={2},
   pages={297--306},
   issn={1259-1734},
   review={\MR{3025137}},
}

\bib{stein}{book}{
   author={Stein, William},
   title={Modular forms, a computational approach},
   series={Graduate Studies in Mathematics},
   volume={79},
   note={With an appendix by Paul E. Gunnells},
   publisher={American Mathematical Society, Providence, RI},
   date={2007},
   pages={xvi+268},
   isbn={978-0-8218-3960-7},
   isbn={0-8218-3960-8},
   review={\MR{2289048}},
   doi={10.1090/gsm/079},
}

\bib{williams-2}{article}{
   author={Williams, Kenneth S.},
   title={Fourier series of a class of eta quotients},
   journal={Int. J. Number Theory},
   volume={8},
   date={2012},
   number={4},
   pages={993--1004},
   issn={1793-0421},
   review={\MR{2926557}},
   doi={10.1142/S1793042112500595},
}

\end{rezabib}

\end{document}